\documentclass[12pt]{amsart}
\usepackage{amsmath,amssymb,amsfonts,enumerate,amsthm,amscd,color, graphics}
\usepackage[bookmarksnumbered, colorlinks, plainpages]{hyperref}
\usepackage[all]{xy}

\numberwithin{equation}{section}

\newtheorem{thm}{Theorem}[section]
\newtheorem{lem}[thm]{Lemma}

\newtheorem{cor}[thm]{Corollary}
\newtheorem{prop}[thm]{Proposition}
\newtheorem{rem}[thm]{Remark}

\newtheorem{exam}[thm]{Example}

\newtheorem{defen}[thm]{Definition}

\newtheorem{defrem}[thm]{Definition and Remark}
\newcommand{\fa}{\mathfrak{a}}
\newcommand{\fb}{\mathfrak{b}}
\newcommand{\fc}{\mathfrak{c}}
\newcommand{\fd}{\mathfrak{d}}

\newcommand{\m}{\mathfrak{m}}
\newcommand{\p}{\mathfrak{p}}
\newcommand{\N}{\mathbb{N}}
\newcommand{\Z}{\mathbb{Z}}
\newcommand{\depth}{\operatorname{depth}}
\newcommand{\minAss}{\operatorname{minAss}}

\newcommand{\Ass}{\operatorname{Ass}}
\newcommand{\Supp}{\operatorname{Supp}}

\newcommand{\cd}{\operatorname{cd}}
\newcommand{\grade}{\operatorname{grade}}
\newcommand{\fgrade}{\operatorname{f-grade}}
\newcommand{\faR}{f_{\fa}^{R_+}}
\newcommand{\fR}{f_{R_+}}
\newcommand{\fg}{finitely generated}

\newcommand{\fG}{finitely graded}
\newcommand{\fGR}{finitely graded $R$-module}
\newcommand{\lM}{linked by $I$ over $M$}
\newcommand{\lR}{linked by $I$ over $R$}

\newcommand{\Ro}{$(R_0,\m_0)$}
\newcommand{\rs}{regular sequence}
\newcommand{\seq}{sequence}
\newcommand{\ab}{$\fa \stackrel{h} \sim_{(I;M)} \fb$}
\newcommand{\aR}{$\fa \stackrel{h} \sim_{(I;M)} R_+$}
\newcommand{\bM}{$\fb M=IM:_M \fa$}
\newcommand{\aM}{$\fa M=IM:_M \fb$}
\newcommand{\IM}{$\fa M \cap \fb M = IM$}
\newcommand{\homo}{homogeneous}
\newcommand{\st}{such that}

\newcommand{\MV}{Mayer-Vietoris}
\newcommand{\ifff}{if and only if}
\newcommand{\rCM}{relative Cohen-Macaulay}
\newcommand{\NAK}{Nakayama Lemma}

\begin{document}
%\bibliographystyle{amsplain}

%\author{}
\address{ }
\email{}

\author{Maryam Jahangiri$^*$}
\address{Department of Mathematics, Faculty of Mathematical Sciences and Computer, Kharazmi University, Tehran,
Iran.}
\email{jahangiri@khu.ac.ir, jahangiri.maryam@gmail.com }
 \thanks{$^*$Corresponding author}

\author{Azadeh Nadali}
%\address{Faculty of Mathematical Sciences and Computer, Kharazmi University, Tehran, Iran}
\email{a.nadali1984@gmail.com}

\author{Khadijeh Sayyari}
%\address{Faculty of Mathematical Sciences and Computer, Kharazmi University, Tehran, Iran}
\email{std\underline{ }sayyari@khu.ac.ir}

\subjclass[2010] {13D45, 13A02, 13C40}
 \keywords{Graded local cohomology modules, linkage of ideals, finiteness dimension, cohomological dimension, relative Cohen-Macaulay.  }

\title[  Graded local cohomology modules with respect to the linked ideals  ]
{ Graded local cohomology modules with respect to the linked ideals }
%%%%%%%%%%%%%%%%%%%%%%%%%%%%%%%%%%%%%%%%%%%%%%%%%%%%%%%%%%%%%%%%%%%%%%%%%%%%%%%%%%%%%%%%%%%%%%%%%%%%%%%%%%%%%%%%%%%%%%%%%%%%%%%%%%%%%%%%%%%%%%%%%%%%%%%%%%%%%%%%%%%%%%%%%%%%

\begin{abstract}%------------------------------------------------------------------------------------------------------------------------------------
Let $R=\oplus_{n\in \N_0}R_n$ be a standard graded ring, $M$ be a finitely generated graded
$R$-module and $R_+:=\oplus_{n\in \N}R_n$ denotes the irrelevant ideal of $R$.
In this paper, considering the new concept of linkage of ideals over a module, we study the graded components $H^i_{\fa}(M)_n$
when $\fa$ is an h-linked ideal over $M$. More precisely, we show that $H^i_{\fa}(M)$ is tame in each of the following cases:
\begin{itemize}
  \item [(i)] $i=\faR(M)$, the first integer $i$ for which $R_+\nsubseteq \sqrt{0:H^i_{\fa}(M)}$;
  \item [(ii)] $i=\cd(R_+,M)$, the last integer $i$ for which $H^{i}_{R_+}(M)\neq 0$, and $\fa=\fb+R_+$ where $\fb$ is an h-linked ideal with $R_+$ over $M$.
\end{itemize}
Also, among other things, we describe the components $H^i_{\fa}(M)_n$ where  $\fa$ is radically h-$M$-licci with respect to $R_+$ of length 2.

\end{abstract}%------------------------------------------------------------------------------------------------------------------------------------

\maketitle
\section{introduction}%%%%%%%%%%%%%%%%%%%%%%%%%%%%%%%%%%%%%%%%%%%%%%%%%%%%%%%%%%%%%%%%%%%%%%%%%%%%%%%%%%%%%%%%%%%%%%%%%%%%%%%%%%%%%%%%%%%%%

Throughout the paper, $R= \bigoplus_{n\in
\N_0}R_n$ is a standard graded Noetherian ring, i.e. $R_0$ is a
commutative Noetehrian ring and $R$ is generated, as an
$R_0$-algebra, by finitely many elements of degree one, $R_+=
\bigoplus_{n\in \N}R_n$ is the irrelevant ideal of $R$ and $\fa$ and $\fb$
are homogeneous  ideals of $R$. Also, $M$  denotes a
finitely generated graded $R$-module.

For $i\in \N_0$, the set of non-negative integers, and $n\in \Z$, the set of integers, let $H^{i}_{\fa}(M)_n$ denotes the
$n$-th component of graded local cohomology module $H^{i}_{\fa}(M)$
of $M$ with respect to $\fa$ (our terminology on local cohomology
comes from \cite{BSH12}). It is well-known that $H^{i}_{R_+}(M)_n$ is a
finitely generated $R_0$-module for all $n\in \Z$ and
$H^{i}_{R_+}(M)_n=0$ for all $n\gg 0$ (\cite[16.1.5]{BSH12}).
The asymptotic behavior of the components $H^i_{R_+}(M)_n$ when $n \rightarrow -\infty$ has been studied by many authors, too. See for example \cite{brodmann(survey)}, \cite{BH}, \cite{jzh} and \cite{thomasmarley}.
But, we know not much about the graded components $H^i_{\fa}(M)_n$ where $\fa$ is an arbitrary \homo ~ ideal of $R$.
Although, there are some studies in this topic, see for example \cite{lcm}, \cite{jz} and \cite{puthenpurakal2017graded}.

In a recent paper (\cite{linkage1}), the authors introduced the concept of linkage of ideals over a module, which is a generalization of its classical concept introduced by Peskine and Szpiro (\cite{peskine1974liaison}).

Let $\fc$ and $\fd$ be ideals of the commutative Noetherian ring $A$ with $1\neq 0$ and $N$ be a \fg ~ $A$-module. Assume that
 $\fc N\neq N \neq \fd N$ and $I \subseteq \fc \cap \fd$ is an ideal
 generated by an $N$-\rs. Then the ideals $\fc$ and $\fd$ are said to be linked by $I$ over $N$, denoted by
$\fc \sim_{(I;N)} \fd$, if $\fc N=IN:_N \fd$ and $\fd N=IN:_N \fc$. In \cite{cohomological3} and \cite{characterization2},
the authors studied some cohomological properties of linked ideals.

In this paper, we consider the above concept in the graded case. More precisely, the \homo ~ ideals $\fa$ and $\fb$ are said to be \homo ly linked (or h-linked) by $I$ over $M$, denoted by \ab,
if $I$ is generated by a \homo ~ $M$-\rs ~ and $\fa$ and $\fb$ are linked by $I$ over $M$.
We consider homogeneously linked ideals and  study some of their  cohomological properties.

This paper is divided into three sections.
In Section 2, we study some basic properties of homogeneously linked ideals.
We Show by examples that if \ab , this doesn't imply that $\fa \cap R_0 \sim_{(I \cap R_0;M_n)} \fb \cap R_0$
for all $n \in \Z$ and vice versa.
Although, in some cases it does (see \ref{exam2} and \ref{rem1}). Due to the importance of irrelevant ideal in a standard graded ring, it is natural to ask  whether a \homo ~ ideal could be h-linked with $R_+$.
We answer this question in some cases, see \ref{x link} and \ref{prop1}.

Section 3, which is the main part of the paper, is devoted to study the graded components of local cohomology modules $H^i_{\fa}(M)_n$ where $\fa$ is an h-linked ideal over $M$. Let $N=\bigoplus_{n \in \Z}N_n$ be a graded $R$-module, $end(N)$ is defined to be the last integer $n$ for which $N_n\neq 0$. As we mentioned above, $end(H^i_{R_+}(M))<\infty$ for all $ i \in \N_0$. Also, in \cite{jz}, it is shown that $end(H^i_{\fa}(M))<\infty$ for all $i \in \N_0$ and all \homo ~ideal $\fa \supseteq R_+$. In Theorem \ref{end2}, we show that if $\fa$ is an h-linked ideal with $R_+$ over $M$, then $end (H^i_{\fa}(M))<\infty$  for any $i\neq \grade(\fa , M)$ and that $end(H^{\grade(\fa , M)}_{\fa}(M))<\infty$ or $H^{\grade(\fa , M)}_{\fa}(M)_n\neq 0$ for all $n\gg 0$ where, $\grade(\fa,M)$ denotes the length of a maximal $M$-\rs ~ in $\fa$.

$\faR(M)$ is defined to be the first integer $i$ for which $R_+\nsubseteq \sqrt{0:H^i_{\fa}(M)}$.
This invariant was studied in \cite{thomasmarley} as well in \cite[\S 9]{BSH12}. Also, a graded $R$-module
$N=\bigoplus_{n \in \Z}N_n$ is said to be tame if $\{ n \in \Z | N_n=0 ~\text{ and}~ N_{n+1}\neq 0\}$ is a finite set.
Tameness of local cohomology modules is one of the most fundamental concepts concerning this modules and attracts lots of interests,
see for example \cite{brodmann(survey)}, \cite{BH}, \cite{lcm}, \cite{jzh} and \cite{rotthaus2005some}. In \cite[2.2]{jzh}, the authors show that $H^{\faR(M)}_{\fa}(M)$ is tame whenever $\fa \supseteq R_+$. In theorem \ref{tame}, we prove it without any restriction on $\fa$.

 In theorem \ref{faR2}, it is shown that if \Ro ~is local, then $f_{\fa +R_+}^{R_+}(M)$ is finite and that $H^{\cd(R_+,M)}_{\fa +R_+}(M)$ is tame in the case where $\fa$ is an h-linked ideal with respect to $R_+$ over $M$.
Here, $\cd(R_+,M)$ is the cohomological dimension of $M$ with respect to $R_+$, that is the last integer $i$ for which $H^{i}_{R_+}(M)\neq 0$.

We keep the notations introduced in the introduction, throughout the paper.

%%%%%%%%%%%%%%%%%%%%%%%%%%%%%%%%%%%%%%%%%%%%%%%%%%%%%%%%%%%%%%%%%%%%%%%%%%%%%%%%%%%%%%%%%%%%%%%%%%%%%%%%%%%%%%%%%%%%%%

\section{Homogenously linked ideals over a module}

We start by the basic concept of the paper.

\begin{defen}
Assume that $\fa M\neq M \neq \fb M$ and $I\subseteq \fa \cap \fb$ be an ideal generated by a \homo ~ $M$-\rs. Then we say that the ideals $\fa$ and $\fb$ are \homo ly linked (or h-linked) by $I$ over $M$, denoted \ab, if \bM ~and \aM. The ideals $\fa$ and $\fb$ are said to be  geometrically h-\lM ~if, in addition, \IM. Also, we say that the ideal $\fa$ is h-linked over $M$ if there exist \homo ~ ideals $\fb$ and $I$ of $R$ such that \ab. $\fa$ is h-$M$-selflinked by $I$ if $\fa \stackrel{h}\sim_{(I;M)} \fa$.
\end{defen}

 \begin{rem}\label{geometrically or not}
  Note that, this definition is a special case of linkage of ideals over a module, studied in \cite{linkage1}. Moreover,
  %as stated in \cite{linkage1},
  if $\fa$ and $\fb$ are h-linked by $I$ over $M$ and $\grade(\fa,M)=t$ then the following statements hold.
  \begin{itemize}
    \item [(i)] If $\fa M \cap \fb M \neq IM$, then $IM:_M (\fa + \fb) \neq IM$. So, $(\fa +\fb) \subseteq Z(M/IM)$, the set of zero divisors of $M/IM$, that results
    $\grade(\fa + \fb , M)=t$.
    \item[(ii)] If \IM ~(i.e. $\fa$ and $\fb$ are geometrically h-linked), then, by \cite[2.9]{cohomological3},
     $\grade(\fa + \fb , M) = t+1$.
  \end{itemize}
\end{rem}

In the next example, we show that there is no bilateral relation between h-linkedness of ideals $\fa$ and $\fb$ by $I$ over $M$ with  linkage of $\fa \cap R_0$ and $\fb \cap R_0$ by $I \cap R_0$ over \homo ~ components of $M$.

\begin{exam}\label{exam2}
Let \Ro ~ be local with $\depth R_0>0$, $\m=\m_0+R_+$ be the \homo ~maximal ideal of $R$ and $x_1,x_2,\ldots, x_s$($s\geq 2$) be a \homo ~ $R$-\rs ~in $\m$. Assume that $1\leq l < s$ such that deg $x_i=0$ for all $1\leq i \leq l$ and deg $x_i\geq 1$ for all $l < i \leq s$.
\begin{itemize}

  \item Set $\fa := (x_1,x_2,\ldots, x_s)$, $I := (x_1,\ldots ,x_l,x^2_{l+1}, \ldots,x_s)$, $\fa_0:=\fa \cap R_0 =(x_1,\ldots, x_l)R_0 $ and $I_0:= I \cap R_0=(x_1,\ldots, x_l)R_0$. So, by \cite[2.2]{linkage1}, $\fa$ is h-$R$-selflinked by $I$. But, since $\fa_0 = I_0$, $\fa_0$ is not $R_n$-selflinked by $I_0$, for all n.

  \item Again, set $\fa := (x_1,x_2,\ldots, x_s)$, $I:= (x^2_1,x_2,\ldots, x_{s-1})$, $\fa_0:=\fa \cap R_0=(x_1,\ldots, x_l)R_0$ and $I_0:= I \cap R_0=(x_1^2,x_2, \ldots, x_l)R_0$. Then, $\grade(\fa,R) = s \neq \grade(I,R)$, so, by \cite[2.6(i)]{linkage1}, $\fa$ is not h-$R$-selflinked by $I$. But, $\fa_0$ is $R_n$-selflinked by $I_0$ for all $n$, using \cite[2.2]{linkage1} and the fact that $x_1^2,x_2,\ldots, x_l$ is an $R_n$-\rs ~for all $n$.
\end{itemize}
\end{exam}

\begin{rem}\label{rem1}
  Assume that \Ro ~ is local, $\fa$ and $\fb$ are generated by elements of degree zero. Then, despite the above example,
   $ \fa \stackrel{h}\sim_{(0;M)} \fb$  \ifff ~ $(\fa \cap R_0) \sim_{(0;M_n)} (\fb \cap R_0)$ for all $n \in \Z$.

\end{rem}

In view of the importance of the irrelevant ideal in a standard graded ring, it is natural to study \homo ~ ideals which are h-linked with $R_+$.

If $R=R_0[x_1 , \ldots , x_n]$ is a polynomial ring graded in the usual way, then $R_+=(x_1 , \ldots , x_n)$ is h-$R$-selflinked by $(x_1^2 ,x_2, \ldots , x_n)$, using \cite[2.2]{linkage1}. In the next example, we find some \homo ~ ideals that are h-linked with $R_+$.
It will be used in the next section, too.

\begin{exam}\label{x link}
\noindent
\begin{enumerate}

  \item Let $R=R_0[x]$ and $x\notin Z(M)$, then $R_+ \stackrel{h}\sim_{((r_0x^t);M)} (r_0x^{t-1})$ for all $r_0 \in R_0 \backslash Z(M)$ which is not unit and all $t\geq 1$.
  \item Let $R=R_0[x,y]$ and
   $r_1x, r_2y$ be an $M$-\rs ~where $r_1,r_2 \in R_0$.
  Then $R_+ \stackrel{h}\sim_{((r_1x^t,r_2y^{t'});M)} (r_1x^t,r_2y^{t'},r_1 r_2 x^{t-1} y^{t'-1})$
  for all $t,t'\geq 1$.
\end{enumerate}
\end{exam}
\begin{proof}
  The first case is obvious. To prove (2), we have to show that
  \begin{itemize}
    \item[(i)] $(x,y)M=(r_1x^t, r_2y^{t'})M:_M (r_1x^t, r_2y^{t'}, r_1r_2 x^{t-1}y^{t'-1})$,
    \item[(ii)] $(r_1x^t, r_2y^{t'}, r_1r_2 x^{t-1}y^{t'-1})M =(r_1x^t, r_2y^{t'})M:_M (x,y)$.
  \end{itemize}
   Since $r_1x,r_2y$ is an $M$-\rs , $x,y$ and $r_1x^i, r_2y^j$ are  $M$-\rs ~for all $i,j \geq 0$.  On the other hand, if $x_1,x_2,\ldots,x_n$ is an $M$-\rs ~ and $m_1,m_2, \ldots, m_n \in M$ such that $m_1x_1+\ldots +m_nx_n=0$, then $m_1 \in (x_2, \ldots x_n)M$.
  \begin{itemize}
    \item[(i)] Assume that $m \in M$ such that $m(r_1x^t, r_2y^{t'}, r_1r_2 x^{t-1}y^{t'-1}) \subseteq (r_1x^t, r_2y^{t'})M$.
    So, $mr_1r_2 x^{t-1}y^{t'-1} \in (r_1x^t, r_2y^{t'})M$, then there exist $m_1,m_2 \in M$ such that
    $(m r_2 y^{t'-1}-xm_1)r_1x^{t-1}-r_2y^{t'}m_2=0$. Thus $m r_2 y^{t'-1}-xm_1 \in r_2y^{t'}M$ and so there exists $m' \in M$ such that $(m-m'y)r_2y^{t'-1}\in Mx$. Hence, $m-m'y \in Mx$ and as a result $m \in (x,y)M$.
    \item [(ii)] Now, let $m \in M$ such that $m(x,y) \subseteq (r_1x^t, r_2y^{t'})M$. So, $mx \in (r_1x^t, r_2y^{t'})M$ and $my \in (r_1x^t, r_2y^{t'})M$. Then, there exist $m_1,m_2 \in M$ such that $(m-r_1x^{t-1}m_1)x-r_2y^{t'}m_2=0$. Thus $m-r_1x^{t-1}m_1 \in r_2y^{t'}M$. So, we have $m=r_1x^{t-1}m_1+r_2y^{t'}m''_1$ where $m''_1 \in M$.
        Same as above, $m=r_1x^{t}m''_2+r_2y^{t'-1}m'_2$ where $m'_2, m''_2 \in M$.
        Hence,
        \begin{equation}\label{equal}
           m=r_1x^{t-1}m_1+r_2y^{t'}m''_1=r_1x^{t}m''_2+r_2y^{t'-1}m'_2.  \qquad \tag{*}
        \end{equation}
       So, $m_1-xm''_2 \in M(r_2y^{t'-1})$ which implies that $m_1= xm''_2+nr_2y^{t'-1}$ for some $n \in M$. As a result, by \eqref{equal},
        $m \in (r_1x^t, r_2y^{t'}, r_1r_2 x^{t-1}y^{t'-1})M$.
  \end{itemize}
\end{proof}

It's natural to ask whether a \homo~ideal which is linked could be an h-linked ideal? In the following, we answer it in a special case.

\begin{rem}\label{rem2}

  Let  $R_0$ be reduced and $R_+$ be a linked ideal by $I$ over $R$.
  As $R_0$ is reduced, $R_+$ is radical.
    Also, in view of \cite[16.1.2]{BSH12}, there exists an ideal, say $I'$, generated by a \homo ~ $R$-\rs ~ of length $t$ in $R_+$, where $t:=\grade R_+$ and $I'\neq R_+$.
  So, using \cite[2.8]{linkage1}, \cite[Theorem 1]{hellus2001set} and \cite[1.4]{singh2007local},

  \begin{align*}
    \Ass R/R_+ &\subseteq  \Ass R/I \cap V(R_+) =\Ass Hom_R(R/R_+,R/I)  \\
     &= \Ass Ext^t_R (R/R_+,R)  =\Ass R/I' \cap V(R_+),
  \end{align*}

   where $V(R_+)$ denotes the set of prime ideals of $R$ containing $R_+$. This implies that $R_+$ is an h-linked ideal by $I'$, by \cite[2.5]{characterization2}.
\end{rem}

\begin{defen}
   Following \cite[2.1]{nagel1994cohomological}, a sequence $x_1,x_2,\ldots ,x_t$
   of homogeneous elements of $\fb$ is said to be a homogeneous $\fb$-filter regular sequence on $M$
    if $x_i \notin \p$ for all $\p \in \Ass(\frac{M}{(x_1,\ldots , x_{i-1})M}) \backslash V(\fb)$ and all
 $i = 1,\ldots,t$.
\end{defen}

Assume that $\fa$ is generated by elements of positive degrees and $\fa \subseteq \fb$. By \cite[1.5]{jz}, if
$\Supp(\frac{M}{\fa M})\nsubseteq V(\fb)$, then all maximal \homo ~ $\fb$-filter \rs s ~in $\fa$ on $M$ have the same finite length, that is denoted by $\fgrade(\fb,\fa,M)$. Also $\fgrade(\fb,\fa,M) := \infty$ whenever $\Supp(\frac{M}{\fa M})\subseteq V(\fb)$.
Note that $\grade(\fa ,M)\leq \fgrade(\fb,\fa,M)$.\\
Moreover, Chu and Gu in \cite[2.4]{lcm} in the case where $\fb=R_+$, show that if $\Supp(\frac{M}{\fa M})\nsubseteq V(R_+)$ then
\[
\fgrade(R_+,\fa,M)= \max \{ i \mid H^j_{\fa}(M)_n=0,~ \text{for all $n\gg 0$ and all $j <i$} \}.
\]

In the following proposition, we consider a polynomial ring and see whether a \homo ~ ideal could be h-linked with $R_+$.

\begin{prop}\label{prop1}
  Let \Ro ~be a regular local ring containing a field of characteristic zero and $R=R_0[x_1, \ldots ,x_t]$ be  the polynomial ring graded in the usual way, that is $\deg (x_i)=1$ for all $i=1,\ldots,t$. Then $R_+$ can't be h-linked with any ideal $\fa \supsetneq R_+$. Moreover, if $\fa \stackrel{h}\sim_{(I;R)} R_+$ and $\fa \nsubseteq R_+$, then $\fa$ and $R_+$ are geometrically h-linked by $I$ over $R$.
\end{prop}

\begin{proof}
  Let $\fa \stackrel{h}\sim_{(I;R)} R_+$ and suppose to the contrary that $R_+ \subsetneq \fa$. Since $t:=\grade R_+ \leqslant f-\grade(\fa,R_+,R)$, so $H_\fa ^t(R)_n$ is a \fg ~ $R_0$-module for all $n \in \Z$, by \cite[1.7]{jz}. Thus $H_\fa ^t(R) =0$, using \cite[8.1]{puthenpurakal2017graded}, that is a contradiction, in view of \cite[2.6(i)]{linkage1}.\\
  Now, assume that $\fa \stackrel{h} \sim_{(I;R)} R_+$ and $\fa \nsubseteq R_+$. By \cite[1.7]{jz} and \cite[8.1]{puthenpurakal2017graded}, $H^i_{\fa+R_+}(R)=0$ for all
   $i \leq f-\grade(\fa+R_+,R_+,R)$. So, $f-\grade(\fa+R_+,R_+,R) \lneq \grade(\fa + R_+)$. Thus $\grade R_+ \lneq \grade(\fa + R_+)$
  and, by \ref{geometrically or not}(i), $\fa \cap R_+ = I$.
  %(i.e. $\fa$ and $R_+$ are \glR ).

\end{proof}

In the following proposition, we study the set $\Ass_{R_0} (M/\fa M)$ where $\fa$ is an h-linked ideal over $M$. It will be used later in the paper,too.

\begin{prop}\label{Ass}
  Assume that $\fa$ and $\fb$ are geometrically h-\lM ~ and $\fb \supseteq R_+$. Then
  \begin{itemize}

    \item [(i)] $\Ass_{R_0}(M/\fa M)=\Ass_{R_0}(M/IM) \bigcap V(\fa_0)$;
    \item [(ii)] $\Ass_{R_0}(M/\fa M)\bigcap \Ass_{R_0}(M/\fb M)=\varnothing$;
    \item [(iii)] $\Ass_{R_0}(M/\fb M) \bigcap V(\fa_0) = \varnothing $.
  \end{itemize}
  The first case also holds if $\fa$ and $\fb$ are just  h-linked and  $\Ass_{R}(M/IM)= \minAss_{R}(M/IM)$ (e.g. $M$ is a Cohen-Macaulay module).
\end{prop}

\begin{proof}
  \begin{itemize}

    \item [(i)] By \cite[2.9]{linkage1} and \cite[Exercise 6.7]{matsumura1989commutative}, $\Ass_{R_0}(M/\fa M)=\{\p \cap R_0~|~\p \in \Ass_{R}(M/IM)
        \\ \bigcap V(\fa)\}$ and that
     $\Ass_{R_0} (M/IM)=\{\p \cap R_0~|~ \p \in \Ass_{R}(M/IM)\}$.
     This implies that  $ \Ass_{R_0}(M/\fa M) \subseteq \Ass_{R_0}(M/IM) \bigcap V(\fa_0)$. \\
     Now, let $\p_0 \in \Ass_{R_0}(M/IM) \bigcap V(\fa_0)$. Then, there exists $\p \in \Ass_{R}(M/IM)$ such that $\p \cap R_0=\p_0$. $\sqrt{0:M + I} =\sqrt {0:M/IM}  \subseteq \p$. Thus, by \cite[2.2]{cohomological3} and the assumption,
      $\sqrt{0:M + \fa \cap R_+ } \subseteq \p$. So, $\p \supseteq \fa$ and, again by \cite[2.9]{linkage1}, $\p \in \Ass_{R}(M/\fa M)$.
    \item  [(ii)] Let $\p_0 \in \Ass_{R_0}(M/\fb M)$ then, by \cite[2.9]{linkage1}, there exists $\p \in \Ass_{R}(M/IM) \bigcap V(\fb)$
     such that $\p \cap R_0=\p_0$ and $\p \notin \Ass_{R}(M/\fa M)$. So, $\p  \notin V(\fa)$. On the other hand, $\p \supseteq \fb \supseteq R_+$ thus $\p_0 \nsupseteq \fa_0$ and by (i), $\p_0 \notin \Ass_{R_0}(M/\fa M)$.
    \item  [(iii)] Follows from (i) and (ii).
  \end{itemize}
\end{proof}

 If we remove the condition $\fb \supseteq R_+$, then the above proposition does not hold any more, as the following example shows.

\begin{exam}\label{exam1}
  Let $\fa$ and $R_+$ be geometrically h-\lR , then $\Ass_{R_0}(R/R_+)\neq \Ass_{R_0}(R/I)$.
\end{exam}

\begin{proof}
    Since $I=\fa \cap R_+$, so $\fa \nsupseteq R_+$.
    Assume that $\Ass_{R_0}(R/R_+)= \Ass_{R_0}(R/I)$, then by \ref{Ass}(i),
    $\Ass_{R_0}(R/\fa) \subseteq \Ass_{R_0}(R/R_+)$
   % $\Ass_{R_0}(R/\fa) \bigcap \Ass_{R_0}(R/R_+)=\Ass_{R_0}(R/\fa)$ so , $\Ass_{R_0}(R/\fa)=\varnothing$
    that is a contradiction, by \ref{Ass}(ii).
\end{proof}

%%%%%%%%%%%%%%%%%%%%%%%%%%%%%%%%%%%%%%%%%%%%%%%%%%%%%%%%%%%%%%%%%%%%%%%%%%%%%%%%%%%%%%%%%%%%%%%%%%%%%%%%%%%%%%%%%%%%%%%%%%%%%%%%

\section{Graded components of $H^i_{\fa}(M)$ where $\fa$ is an h-linked ideal}

In this section, which is the main part of the paper, we study the graded components of $H^i_{\fa}(M)$ where $\fa$ is h-linked with the irrelevant ideal $R_+$ over $M$.

For a graded $R$-module $N=\bigoplus_{n \in \Z}N_n$, set
\[
end(N):=sup \{n \in \Z |~ N_n \neq 0  \}.
\]
Note that $end(N)$ could be $\infty$ and that the supremum of the empty set is to be taken as $-\infty$.

The following lemma, which consider a case where $end(H^i_{\fa}(M))<\infty$, will be used later in the paper, too.

\begin{lem}\label{end3}
   Let $t \in \N_0$ and assume that $end(H_{\fa}^i(M))<\infty$ for all $i\neq t$. Then
    for all $n\gg 0$ and all $i \in \N_0$,
    \[
    H_{\fa +\fb} ^i(M)_n\cong \left \{ \begin{array}{cc}
                                                           H_{\fb}^{i-t}(H_{\fa}^t(M))_n, & \quad i\geq t\\
                                                          0 & \quad i \lneq t.

\end{array} \right.
    \]
 \end{lem}

 \begin{proof}
   We have the following convergence of spectral \seq s, by  \cite[11.38]{rotman-spectral},
   \[
   (E_2 ^{i,j})_n=H_{\fb}^i(H_{\fa}^j(M))_n \stackrel{i} \Rightarrow H_{\fa + \fb}^{i+j}(M)_n.
   \]
   Since $end(H_{\fa}^j(M)) < \infty $ for all $j \neq t$, $H_{\fa}^j(M)$ is $R_+$-torsion for all $j \neq t$.
   So, by \cite[2.1.9]{BSH12},
   \[
   H_{\fb}^i(H_{\fa}^j(M))\cong H_{\fb +R_+}^i(H_{\fa}^j(M)) \cong H_{\fb_0R}^i(H_{\fa}^j(M)) \qquad \text{for all $j \neq t$
   and all $i\geq 0$},
   \]
   where $\fb_0:=\fb \cap R_0$. Hence, by \cite[14.1.12]{BSH12} and the assumption, $(E_2 ^{i,j})_n=H_{\fb_0}^i(H_{\fa}^j(M)_n)=0$ for all $j\neq t$ and all $n\gg 0$. As a result, $H_{\fa +\fb} ^{i}(M)_n\cong (E_2^{i-t,t})_n $ for all $i\geq t$ and that
    $H_{\fa +\fb} ^{i}(M)_n=0$ for all $i \lneq t$, when $n \gg 0$.
 \end{proof}

The following corollary, which is immediate by the above lemma, generalizes \cite[1.1]{jz}.
 \begin{cor}\label{end6}
 Let $end(H^i_{\fa}(M))<\infty$ for all $i \in \N_0$. Then for any  \homo ~ ideal $\fb \supseteq \fa$, $end(H_{\fb} ^i(M))<\infty$ for all $i \in \N_0$.
 \end{cor}

The following lemma  will be used several times in the paper.

\begin{lem}\label{supp}
  Let $\fa$ be  linked by $I$ over $M$. Then $\Supp H^t_{\fa}(M)=\Supp M/\fa M$, where $t:=\grade(\fa , M)$.
\end{lem}

\begin{proof}
  By \cite[2.8]{linkage1}, \cite[Theorem 1]{hellus2001set} and \cite[1.4]{singh2007local},
  \[
   \Ass M/\fa M \subseteq \Ass M/IM \cap V(\fa)=\Ass Hom_R(R/\fa , M/IM)= \Ass Ext^t_R(R/\fa , M)=\Ass H^t_{\fa}(M).
  \]
  On the other hand, $\Supp H^t_{\fa}(M) \subseteq \Supp M/\fa M$, which proves the claim.
\end{proof}

In the following, we show some equivalent conditions for $end(H^i_{\fa}(M))<\infty$, where $\fa$ is an h-linked ideal over $M$.

\begin{prop}\label{end4}
  Let $\fa$ be an h-linked ideal by $I$ over $M$ with $\grade(\fa , M)=t$. Then the following statements are equivalent.
  \begin{itemize}
    \item[(i)] $end(H^i_{\fa}(M))<\infty$ for all $i \in \N_0$,
    \item [(ii)] $end(H^t_{\fa}(M))<\infty$,
    \item [(iii)] $\Supp M/\fa M \subseteq V(R_+)$.
  \end{itemize}
  Also, if $\fa \stackrel{h}\sim_{(I;M)}\fb$ and one of the above conditions holds, then
   \begin{equation*}
H_{\fb}^i (M)_n \cong \left \{ \begin{array}{cc}
                                                           0 &  i\neq t\\
                                                           H_I^t(M)_n &  i=t,

\end{array} \right.
\end{equation*}

   for all $i$ and all $n\gg 0$.
\end{prop}
\begin{proof}
   "$(ii) \Rightarrow (iii)$" Since $end(H^t_{\fa}(M))< \infty$, $ H^t_{\fa}(M)$ is $R_+$-torsion and $\Ass H^t_{\fa}(M)\subseteq V(R_+)$. So, the result follows from \ref{supp}.\\
   "$(iii)\Rightarrow (i)$" Since $\sqrt{\fa+ 0:M}\supseteq \sqrt{R_+}$, using \cite[2.1.9]{BSH12}, \cite[16.1.5(ii)]{BSH12} and \ref{end6}, the statement holds.

  The last statement follows from \cite[2.2(i)]{cohomological3}, \ref{end6} and the following \homo ~ \MV~ sequence
  \[
  \ldots \longrightarrow H^i_{\fa +\fb}(M)\longrightarrow H^i_{\fa}(M)\oplus H^i_{\fb}(M)\longrightarrow H^i_I(M) \longrightarrow
  H^{i+1}_{\fa +\fb}(M) \longrightarrow \ldots ~.
  \]
\end{proof}

Note that if $\fb \subseteq R_+$ and one of the above conditions holds, then
%$\fa M \cap \fb M \neq IM$ .
 $\fa$ can't be geometrically h-linked with $\fb$.
 Otherwise, by \ref{end4}(iii), $H^i_{\fa}(M)\cong H^i_{\fa +R_+}(M)$ for all $i$, so $\grade(\fa , M)=\grade(\fa +\fb,M)$, that is a contradiction by \ref{geometrically or not}.

\begin{defen}
\noindent
\begin{itemize}
  \item We say that the ideal $I$ is generated by an $M$-\rs ~ under radical if there exists an $M$-\rs ~ $\underline{x}=x_1, \ldots , x_t$ \st~ $\sqrt{I+0:M}=\sqrt{\underline{x}+0:M}$.
  \item $\fa$ and $\fb$ are said to be in an h-$M$-linkage class of length $n$ if there exist $n \in \N$ and \homo ~ ideals $\fa_1, \ldots, \fa_n$ and $I_1, \ldots, I_n$ \st ~ $\fa =\fa_0\stackrel{h}\sim_{(I_1;M)} \fa_1 \stackrel{h}\sim \ldots \stackrel{h}\sim_{(I_n;M)} \fa_n =\fb$. If, in addition, $\fb$ is generated by a \homo ~ $M$-\rs ~ under radical, then $\fa$ is called radically h-$M$-licci with $\fb$ of length  $n$.
\end{itemize}

\end{defen}

\begin{cor}\label{end5}
  If $\fa \stackrel{h}\sim_{(I;M)}\fb$ and $t:=\grade(\fa,M)$, then the following statements are equivalent.
  \begin{itemize}
    \item [(i)] $\max \{ end(H^t_\fa(M)),end(H^t_\fb(M))\}<\infty$,
    \item [(ii)] $\Supp M/IM \subseteq V(R_+)$.
  \end{itemize}
  In particular, if $\fb=R_+$ and $end(H^t_\fa(M)) <\infty$, then $\fa$ is radically h-$M$-licci with $R_+$ of length 1.
\end{cor}

\begin{proof}
  The results follow from \cite[2.6(iii)]{linkage1}, \ref{end4} and the fact that $end(H^i_{R_+}(M))<\infty$ for all $i \in \N_0$.
\end{proof}

  The above corollary shows, if $\fa \stackrel{h}\sim_{(I;M)}\fb$, then $end(H^i_\fa(M))$ or $end(H^i_\fb(M))$ is infinite for some $i \in \N_0$ \ifff~ $\Supp M/IM \nsubseteq V(R_+)$.

\begin{prop}\label{end7}
  Assume that $\fa \stackrel{h}\sim _{(I;M)} \fb \stackrel{h}\sim_{(J;M)} \fc$
     and $end (H^t_I(M)) <\infty$, where $t:=\grade(I,M)$. Then
     \begin{equation*}
H_{\fc}^i (M)_n =\left \{ \begin{array}{cc}
                                                           0 &  i\neq t\\
                                                           H_J^t(M)_n &  i=t,

\end{array} \right.
\end{equation*}
 for all $i$ and all $n\gg 0.$
   \end{prop}
   
\begin{proof}
  
    Since $end (H^t_I(M)) <\infty$, $H^t_I(M)$ is $R_+$-torsion. So, by \cite[Theorem 1]{hellus2001set}, \cite[1.4]{singh2007local} and \cite[2.6]{linkage1}, $\Supp (M/\fb M) \subseteq V(R_+)$, in other words $R_+ \subseteq \sqrt{\fb + 0:M}$. Thus $H^i_{\fb}(M)_n \cong H^i_{\fb+R_+}(M)_n=0$ for all $i$ and all $n\gg 0$, by \ref{end6}.  Also, in view of \cite[2.2]{cohomological3}, we have the following \homo ~ \MV ~ \seq
     \[
      \ldots \longrightarrow H^i_{\fb +\fc}(M)\longrightarrow H^i_{\fb}(M)\oplus H^i_{\fc}(M)\longrightarrow H^i_J(M) \longrightarrow H^{i+1}_{\fb +\fc}(M) \longrightarrow \ldots ~.
  \]
  This, in conjunction with \ref{end6} and \cite[2.6]{linkage1}, follows the result.

\end{proof}

Using \ref{end7}, we can describe the components $H^i_{\fa}(M)_n$ where $\fa$ is radically h-$M$-licci with $R_+$ of length 2, as follows:

\begin{cor}
  If $\fa$ is radically h-$M$-licci with $R_+$ of length 2, i.e. $\fa \stackrel{h}\sim_{(I;M)} \fb \stackrel{h}\sim_{(J;M)} R_+$ and $\sqrt{R_+ + 0:M}=\sqrt{J+0:M}$, then
   \begin{equation*}
H_{\fa}^i (M)_n =\left \{ \begin{array}{cc}
                                                           0 &  i\neq t\\
                                                           H_I^t(M)_n &  i=t,

\end{array} \right.
\end{equation*}

 for all $i$ and all $n\gg 0$ where $t:=\grade(I,M)$.
\end{cor}
\begin{proof}
 By hypothesis $H^t_J(M)\cong H^t_{R_+}(M)$. So, using \cite[16.1.5]{BSH12} and \ref{end7}, the statement holds.
\end{proof}

$M$ is called \rCM~ with respect to $\fa$ of degree $n$ if $H^i_{\fa}(M)=0$ for all $i\neq n$.

\begin{thm}\label{end2}

   If \aR~  and $t:=\grade(R_+,M)$, then $end(H_{\fa}^i(M)) < \infty$ for all $i\neq t$ and  $end(H_{\fa}^t(M))<\infty$ or $H_{\fa}^t(M)_n\neq 0$ for all $n\gg0$. \\
  In a special case, $H_{\fa}^t(M)_n$ is a \fg ~$R_0$-module for all $n \in \Z$.

\end{thm}

\begin{proof}
  The case $i\neq t$ follows from \ref{end4} and \cite[16.1.5(ii)]{BSH12}. Also, we have $H_{\fa}^t(M)_n \cong H_I^t(M)_n$ for all $n\gg 0$. We consider two cases:
  \begin{itemize}
    \item[case 1:] Let $\Supp (M/IM) \nsubseteq V(R_+)$. By \ref{end5} and \cite[16.1.5(ii)]{BSH12}, $end (H_{\fa}^t(M))=\infty$. Now, we prove, by induction on $t$, that $H_{\fa}^t(M)_n\neq 0$ for all $n\gg 0$.
        \\ If $t=0$, then  $\Gamma_\fa (M)_n=\Gamma_{\underline{0}}(M)_n=M_n$ for all $n\gg0$. On the other hand, by \cite[Theorem 1]{kirby1973artinian}, $R_1M_n=M_{n+1}$ for all $n\gg 0$, thus $\Gamma_{\fa}(M)_n\neq 0$ for all $n\gg 0$.\\
        Let $t>0$ and assume, inductively, that the claim holds for $t-1$. Let $I=(x_1,x_2,\ldots ,x_t)$
        %($x_1,x_2,\ldots ,x_t is \homo ~M$-\seq)
        and $deg(x_1)=l$. Now, the \homo ~ exact \seq ~
         \[ 0\longrightarrow M \stackrel{.x_1} \longrightarrow M(l) \longrightarrow (M/x_1M)(l) \longrightarrow 0 \]
          and \cite[2.6(i)]{linkage1} yield the following exact \seq ~ of $R_0$-modules for all $n \in \Z$,
         \begin{equation*}
           0\longrightarrow H_{\fa}^{t-1}(M/x_1M)_{n+l} \longrightarrow H_{\fa}^t(M)_{n} \stackrel{.x_1} \longrightarrow H_{\fa}^t(M)_{n+l} \longrightarrow \ldots ~.
         \end{equation*}
        Since $x_1 \in I$, $\fa/(x_1) \stackrel{h}\sim _{(I/(x_1);M/x_1M)} R_+/(x_1)$ and, by the inductive hypothesis, $H_{\frac{\fa}{(x_1)}}^{t-1}(M/x_1M)_n \neq 0$ for all $n \gg 0$. Hence, $H_{\fa}^t(M)_n \neq 0$ for all $n \gg 0$.
    \item[case 2:] Now, assume that $\Supp (M/IM) \subseteq V(R_+)$. Then $ \sqrt{0:M + I} = \sqrt{0:M + R_+}$ and $H_I^i(M) \cong H_{R_+}^i(M)$ for all $i$. So, by \cite[16.1.5]{BSH12}, we have
    \begin{enumerate}
      \item $end(H_I^t(M))< \infty$,
      \item $H_I^t(M)_n$ is \fg ~ $R_0$-module for all $n$,
      \item $M$ is \rCM ~with respect to $R_+$ of degree $t$.
    \end{enumerate}
    Therefore, in view of \ref{end4}, $end(H_\fa^t(M))< \infty$. Also, using \cite[3.4]{nagel1994cohomological}, there are
    \homo ~ isomorphisms
     \begin{equation}\label{iso3}
       H^i_{\fa}(M) \cong H^{i-t}_{\fa}(H^t_I(M)) \cong H^{i-t}_{\fa +R_+}(H^t_I(M))
         \cong H^i_{\fa +R_+}(M) \qquad  \text{for all}~ i\geq t.  \tag{3.1}
     \end{equation}

    Using \cite[2.6]{linkage1}, $t  \leq \fgrade(\fa_0 + R_+,R_+,M)$. Therefore, $H_{\fa_0 + R_+} ^t(M)_n$ is a \fg ~$R_0$-module for all $n \in \Z$, by \cite[1.7]{jz}. As a result, by \eqref{iso3}, $H_{\fa}^t(M)_n$ is a \fg ~$R_0$-module for all $n \in \Z$.

  \end{itemize}
\end{proof}

 As wee have seen in the proof of \ref{end2}, if $R_+$ is h-linked by $I$ over $M$ and $\Supp (M/IM) \subseteq V(R_+)$ then $M$ is \rCM ~ with respect to $R_+$. However, the converse does not hold any more, as the following example shows. Although, it does in some special cases, see Proposition \ref{rCM} and \ref{CM}.

\begin{exam}
  Assume that \Ro ~is a domain and $\dim R_0=2$. Set $R=R_0[x]$. So, there exists a  non zero prime ideal $\p_0$ of $R_0$ such that $ \p_0 \subsetneq \m_0$. By \ref{x link}, for any $0\neq r_0 \in \p_0$, $(r_0) \stackrel{h}\sim _{((r_0x);R)} (x)$  and
  $\Supp(R/(r_0x))= V(r_0) \bigcup V(x) \nsubseteq V(x)$,
  %($r_0 \in \p_0[x]$, but $x \notin \p_0[x]$)
  while $R$ is \rCM ~with respect to $(x)$ of degree 1.
\end{exam}

The following proposition considers a case where the irrelevant ideal can be generated by an $M$-\rs ~under radical.

\begin{prop}\label{rCM}
  Let \Ro ~be local and $M$ be \rCM ~with respect to $R_+$ of degree $t$. Then there exists a maximal \homo ~ $M$-\rs ~$I$ in $R_+$ such that $\Supp(M/IM) \subseteq V(R_+)$. In other words, $R_+$ can be generated by a \homo ~ $M$-\rs ~ under radical.
\end{prop}

\begin{proof}
  Assume that $t=0$. By \cite[2.3]{brodmann(survey)}, $\dim M/\m_0M =0$. Therefore,
   $M/\m_0M$ is Artinian and $end(M/\m_0M)<\infty $, using \cite[Theorem 1]{kirby1973artinian}. Hence, by \NAK, $end(M)<\infty$. This implies that $M$ is $R_+$-torsion and that $\Supp M \subseteq V(R_+)$.
  \\ Now, let $t>0$ and assume inductively that the statement holds for $t-1$. As $t>0$, $R_+ \nsubseteq \bigcup _{\p \in (Min Ass(M/\m_0M) \cup Z(M))} \p$, where $Z(M)$ denotes the set of zero divisors on $M$. So, by \cite[16.1.2]{BSH12}, there exists a \homo~ element
  \begin{equation*}
  x \in R_+ \backslash \bigcup_{Min Ass(M/\m_0M) \cup Z(M)} \p.
  \end{equation*}
     Therefore, $\dim \frac{M/xM}{\m_0(M/xM)}=t-1= \grade(R_+,M/xM) $.
   In other words, using \cite[2.3]{brodmann(survey)}, $M/xM$ is \rCM~ with respect to $R_+$ of degree $t-1$ and, by the induction hypothesis, there is a maximal \homo ~ $M/xM$-\rs ~ $I'$ in $R_+$ such that $\Supp\frac{M}{(I'+<x>)M} = \Supp \frac{M/xM}{I'(M/xM)} \subseteq V(R_+)$. Now, the result follows by induction.
\end{proof}

 \begin{prop}\label{CM}
 Let $I$ be an ideal generated by a maximal $M$-\rs ~ in $R_+$. Then the following statements hold.
 \begin{itemize}
   \item [(i)] If $R_0$ is a field then $\Supp M/IM \subseteq V(R_+)$ \ifff~ $M$ is a Cohen-Macauly module.
   \item [(ii)]  $\Supp M/IM \nsubseteq V(R_+)$  \ifff~ $\grade(I,M) = \fgrade(R_+,I,M)$.
 \end{itemize}
 \end{prop}

 \begin{proof}
 \begin{itemize}
   \item [(i)] If $\Supp M/IM \subseteq V(R_+)$ then $\sqrt{0:M+I}=\sqrt{0:M+ R_+}$ and $H^i_{I}(M)\cong H^i_{R_+}(M)$ for all $i \in \N_0$. It follows from \cite[6.2.9]{BSH12} and the fact that $I$ is generated by $M$-\rs , that $M$ is Cohen-Macaulay. Now, assume that $M$ is Cohen-Macaulay. Then $\dim M=\grade(I,M)$. Hence, $\dim M/IM = 0$ and  $\Supp M/IM \subseteq \{R_+\}$.
   \item [(ii)] The result follows from \cite[2.4]{lcm} and \cite[3.3.1]{BSH12}.
 \end{itemize}

 \end{proof}

\begin{defrem}\label{def}
 \noindent
 \begin{itemize}
   \item [(i)] Let  $N=\bigoplus_{n \in \Z}N_n$ be a graded $R$-module. Then following \cite{thomasmarley},
   \begin{itemize}
     \item [$\bullet$] $N$ is called finitely graded if  $N_n=0$  for all but finitely many $n \in \Z$;
     \item [$\bullet$] $g_{\fa}(N) := sup\{k \in \N_0 | H_{\fa}^i(N)$ is finitely graded for all $i < k\}$;
     \item [$\bullet$] $\faR (N) := sup\{k \in \N_0 | R_+ \subseteq \sqrt{0:H_{\fa}^i(N)}$ for all $i < k\}$;
     \item [$\bullet$] $N$ is called tame, if the set $\{n\in \Z | N_n=0, N_{n+1}\neq 0\}$ is finite.
   \end{itemize}
  Note that, by \cite[2.3]{thomasmarley}, if $N$ is \fg , then $g_{\fa}(N)=\faR (N)$.

  \item [(ii)] Let \aR ~and $\Supp M/IM \nsubseteq V(R_+)$, then using  \cite[2.6(i)]{linkage1}, \ref{end2} and (i),  we have $f_{\fa}^{R_+}(M)=\grade(R_+,M)$.

\end{itemize}
\end{defrem}

In \cite[2.2]{jzh}, the authors studied tameness of $H^{\faR (M)}_{\fa}(M)$ under the assumption that $\fa \supseteq R_+$. In the following Theorem, we consider this problem without any restriction on $\fa$. Although, the proof is a modification of \cite[2.2]{jzh}, we bring it here for the reader's convenience.  It will be used later in the paper, too.

\begin{thm}\label{tame}
  Let \Ro ~be local and $f_{\fa}^{R_+}(M) < \infty$. Then $H^{f_{\fa}^{R_+}(M)}_{\fa}(M)$ is tame.
\end{thm}
\begin{proof}
  Let $y$ be an indeterminate. Set $R'_0=R_0[y]_{\m_0[y]}$, $R'=R\bigotimes_{R_0} R'_0$ and $M'=M\bigotimes_{R_0}R'_0$.
  $R'_0$ is a faithfully flat $R_0$-algebra so, by \cite[16.2.2(iv)]{BSH12},
  $H^i_{\fa}(M)_n\bigotimes_{R_0}R'_0\cong H^i_{\fa R'}(M')_n$ for all $i \in \N_0$ and all $n \in \Z$. This results
  $f_{\fa}^{R_+}(M)=f_{\fa R'}^{R'_+}(M)$. Thus, replacing $R$ by $R'$, we can assume that $R_0/\m_0$ is an infinite filed. Now, we prove the assertion by induction on $f:=f_{\fa}^{R_+}(M)$.\\
  Let $f=0$. By \cite[Theorem 1]{kirby1973artinian}, $\Gamma_{\fa}(M)_n=0$ for all $n \ll 0$ and
  $R_1\Gamma_{\fa}(M)_n=\Gamma_{\fa}(M)_{n+1}$ for all $n\gg 0$ that makes $\Gamma_{\fa}(M)_n\neq 0$ for all $n \gg 0$.
  Now, assume that $f\geq 1$ and the result has been proved for $f-1$. \\
   $\Gamma_{R_+}(M)$ is a \fGR ~ so, using \cite[2.2]{thomasmarley}, $H^i_{\fa}(\Gamma_{R_+}(M))$ is \fG ~ for all $i \in \N_0$. Hence, the exact \seq
  \[
  \ldots \longrightarrow H^i_{\fa}(\Gamma_{R_+}(M))_n \longrightarrow H^i_{\fa}(M)_n \longrightarrow
  H^i_{\fa}(M/\Gamma_{R_+}(M))_n \longrightarrow H^{i+1}_{\fa}(\Gamma_{R_+}(M))_n\longrightarrow \ldots
  \]
  implies that $H^i_{\fa}(M)_n \cong H^i_{\fa}(M/\Gamma_{R_+}(M))_n$ for all $i \in \N_0$ and all $n \in \Z \setminus X$, where $X$ is a finite set. So, replacing $M$ with $M/\Gamma_{R_+}(M)$, we can assume that $M$ is $R_+$-torsion free and that there exists a \homo ~ element $x \in R_1$ which is a non-zero deviser on $M$, by \cite[16.1.4(ii)]{BSH12}. Now, consider the \homo ~ exact \seq
  \[
  0\longrightarrow M \stackrel{x} \longrightarrow  M(1) \longrightarrow (M/xM)(1) \longrightarrow 0.
  \]
  It yields the following exact \seq ~of $R_0$-modules
  \[
  \ldots \longrightarrow H^{i-1}_{\fa}(M)_{n+1} \longrightarrow H^{i-1}_{\fa}(M/xM)_{n+1} \longrightarrow H^{i}_{\fa}(M)_{n}
  \stackrel{x} \longrightarrow  H^{i}_{\fa}(M)_{n+1} \longrightarrow \ldots ~.
  \]
  Therefore, by \ref{def}, $\faR(M/xM) \geq \faR(M)-1$.\\
   If $\faR(M/xM) \gneq f-1$, by \ref{def},
  $0 \longrightarrow H^{f}_{\fa}(M)_{n} \stackrel{x} \longrightarrow  H^{f}_{\fa}(M)_{n+1} $ is exact for all $n\gg 0$ and all $n\ll 0$ that shows $H^f_{\fa}(M)$ is tame.\\
   Also, if $\faR(M/xM) = f-1$, again by \ref{def}, we have the following exact \seq
  \[
  0\longrightarrow H^{f-1}_{\fa}(M/xM)_{n+1} \longrightarrow H^{f}_{\fa}(M)_{n}
  \stackrel{x} \longrightarrow  H^{f}_{\fa}(M)_{n+1}
  \]
  for all $n\gg0$ and all $n\ll 0$. But, by induction, $ H^{f-1}_{\fa}(M/xM)$ is tame, which requires $H^{f}_{\fa}(M)$ is tame.

\end{proof}

Regards $t:=\grade(\fa , M)\leq \faR(M)$ and \ref{def}, $H^t_{\fa}(M)$ is tame.

Let $N$ be an $R$-module. The cohomological dimension of $N$ with respect to $\fa$ is defined to be
\[
\cd(\fa,N):=sup \{i \in \Z | H^i_{\fa}(N)\neq 0 \}.
\]

The following Theorem considers tameness of $H^{\cd(R_+,M)}_{\fa+R_+}(M)$ with some restrictions on $M$ or linkedness of $\fa$ with $R_+$.

\begin{thm}\label{faR2}
  Let \Ro ~be local  and $\cd(R_+,M)\neq 0$. Then $f_{\fa +R_+}^{R_+}(M) <\infty$ and  $H^{\cd(R_+,M)}_{\fa+R_+}(M)$ is tame
  in each of the following cases:

  \begin{itemize}
    \item [(i)] $M$ is \rCM ~ with respect to $R_+$;
    \item [(ii)] \aR.
  \end{itemize}
\end{thm}

\begin{proof}
\begin{itemize}

  \item [(i)] Let $t:=\cd(R_+,M)=\grade(R_+,M)$.
  Since $H^i_{R_+}(M)_n$ is a \fg ~ $R_0$-module for all $i$ and all $n$ (\cite[16.1.5]{BSH12}), by \cite[2.6]{rotthaus2005some}, $\fa_0 H^t_{R_+}(M)_n\neq H^t_{R_+}(M)_n$ for all $n\ll 0$, where $\fa_0:=\fa \cap R_0$. Therefore, using \cite[6.2.7]{BSH12},
  \begin{equation}\label{grade}
   \text{for all $n\ll 0$ there exists $k_n \in \N_0$ \st} ~H^{k_n}_{\fa_0}(H^t_{R_+}(M)_n)\neq 0 \tag{3.2}.
  \end{equation}

   Now, considering the following Grothendieck's spectral sequence (\cite[11.38]{rotman-spectral})
  \[
  E_2^{i,j}=H^i_{\fa}(H^j_{R_+}(M)) \stackrel{i} \Rightarrow H^{i+j}_{\fa +R_+}(M),
  \]
  we have $E_2^{i,j}=0$ for all $i$ and all $j\neq t$. This implies that
  \begin{equation}\label{iso6}
    H^i_{\fa}(H^t_{R_+}(M)) \cong H^{i+t}_{\fa +R_+}(M) \qquad \text{for all}~ i \in \N_0. \tag{3.3}
  \end{equation}

 On the other hand, by  \cite[14.1.12]{BSH12} and the fact that $H^t_{R_+}(M)$
   is $R_+$-torsion, we have
   \begin{equation}\label{iso5}
   H^i_{\fa}(H^t_{R_+}(M))_n \cong H^i_{\fa_0+R_+}(H^t_{R_+}(M))_n \cong H^i_{\fa_0R}(H^t_{R_+}(M))_n \cong
   H^i_{\fa_0}(H^t_{R_+}(M)_n) \tag{3.4}
   \end{equation}
    for all $i$ and all $n $. So, if $f_{\fa +R_+}^{R_+}(M) =\infty$, then by \eqref{iso5}, \eqref{iso6} and \ref{def},
    \[
    H^i_{\fa_0}(H^t_{R_+}(M)_n)\cong H^{i+t}_{\fa +R_+}(M)_n=0
   \]
    for all $i$ and all $n\ll 0$ that is a contradiction with \eqref{grade}. Therefore, $f_{\fa +R_+}^{R_+}(M) < \infty$.
    In addition, $t=\grade(R_+,M)\leq \grade(\fa+R_+,M)\leq f_{\fa +R_+}^{R_+}(M)$. So, by \ref{def} and \ref{tame}, $H^t_{\fa+R_+}(M)$ is tame.
  \item [(ii)]
  Set $t:=\grade(R_+,M)$. By the \homo ~\MV ~ \seq~ and \cite[2.2]{cohomological3}, we have the  \homo ~exact \seq
  \begin{equation*}
    \ldots \longrightarrow H_I^{i-1}(M) \longrightarrow H_{\fa + R_+} ^i(M) \longrightarrow H_{\fa}^i(M) \oplus H_{R_+}^i(M) \longrightarrow H_I^i(M) \longrightarrow H_{\fa +R_+} ^{i+1}(M) \longrightarrow \ldots ~.
  \end{equation*}
  It yields
  \begin{equation}\label{iso2}
    H_{\fa + R_+} ^i(M)\cong H_{\fa}^i(M) \oplus H_{R_+}^i(M) \qquad \text{for all $i \gneq t+1$} \tag{3.5}
  \end{equation}
  and, by \cite[6.2.7]{BSH12}, the exact \seq
  \begin{equation}\label{exact}
    0 \longrightarrow H_{\fa + R_+} ^t(M) \longrightarrow H_{\fa}^t(M) \oplus H_{R_+}^t(M) \longrightarrow H_I^t(M) \longrightarrow H_{\fa +R_+} ^{t+1}(M) \longrightarrow
    H_{\fa}^{t+1}(M) \oplus H_{R_+}^{t+1}(M) \longrightarrow 0 \tag{3.6}.
  \end{equation}
  Therefore, in view of \eqref{iso2}, \eqref{exact} and  \cite[2.6]{rotthaus2005some},
  $H^{\cd(R_+,M)}_{\fa +R_+}(M)_n\neq 0$ for all $n\ll 0$. So, by \cite[1.1]{jz}, $H^{\cd(R_+,M)}_{\fa+R_+}(M)$ is tame and, using \ref{def}, $f_{\fa +R_+}^{R_+}(M) \leq \cd(R_+,M) <\infty$.

  \end{itemize}
\end{proof}

\begin{rem}\label{faR3}
\noindent
\begin{itemize}
  \item [(i)] Here is another situation for the finiteness of $\faR (M)$.
  Assume that $\fa$ and $R_+$ are geometrically h-linked over $M$. Then, by \cite[2.9(iii)]{linkage1}, $\Supp (M/\fa M) \nsubseteq V(R_+)$ and there exists a \homo ~ prime ideal $\p \in \Supp M \cap V(\fa) \backslash V(R_+)$.
   Hence, $\p +\fa +R_+ = \p+ R_+ =\p \cap R_0 +R_+ \neq R$. Therefore, by \cite[9.3.7]{BSH12},
  \[
  f^{R_+}_{\fa +R_+}(M) \leq \depth M_{\p} + ht (\fa +R_+ +\p)/\p <\infty .
  \]

  \item [(ii)] Assume that $\fb \supseteq \fa$. Then $\fb$ can be represented as $\fb=\fa +(b_1, \ldots , b_s)$ for some \homo ~ elements $\fb_1, \ldots, \fb_s \in R$. Using \cite[14.1.11]{BSH12} and induction on $s$, one can see that $\faR (M) \leq f_{\fb}^{R_+}(M)$.
\end{itemize}
\end{rem}

The following proposition presents possibilities for $\faR(M)$ and $\fR(M)$ in the case where $\fa$ is h-linked with $R_+$.

\begin{prop}\label{faR}
  Let \aR , then $\faR(M)$, $\fR(M) \in \{\grade(R_+,M), f_{\fa +R_+}^{R_+}(M)\}$.
\end{prop}
\begin{proof}
  Set $t:=\grade(R_+,M)$. By \cite[2.6(i)]{linkage1} and \cite[6.2.7]{BSH12}, $t \leq \faR(M) , \fR(M)$. Also, by \ref{faR3}(ii),
   $\faR(M) , \fR(M) \leq f_{\fa +R_+}^{R_+}(M)$.\\
  If $f_{\fa +R_+}^{R_+}(M) \leq t+1$, the result follows. So, let $f_{\fa +R_+}^{R_+}(M) \gneq t+1$. By \eqref{iso2} and \eqref{exact}, for all $i\geq t+1$, $H^i_{\fa}(M)$ and $H^i_{R_+}(M)$ are \fG ~\ifff ~ $H^i_{\fa+R_+}(M)$ is \fG ~ and this proves the claim.
  \end{proof}

 %%%%%%%%%%%%%%%%%%%%%%%%%%%%%%%%%%%%%%%%%%%%%%%%%%%%%%%%%%%%%%%%%%%%%%%%%%%%%%
 %%%%%%%%%%%%%%%%%%%%%%%%%%%%%%%%%%%%%%%%%%%%%%%%%%%%%%%%%%%%%%%%%%%%%%%%%%%%%%


\begin{thebibliography}{99}

\bibitem{brodmann(survey)} Brodmann M. Asymptotic behaviour of cohomology: tameness, supports and associated primes. In: Ghorpade S, Srinivasan H, Verma J (editors). Commutative Algebra and Algebraic Geometry. American Mathematical Society, 2005, pp. 31-61.

 \bibitem{BH} Brodmann M, Hellus M. Cohomological patterns of coherent sheaves over projective schemes. Journal of Pure and Applied Algebra 2002; 172:
165-182. 

\bibitem{BSH12} Brodman M, Sharp RY. Local cohomology: An algebraic introduction with geometric applications. NY, USA: Cambridge University Press, 2012.

\bibitem{lcm} Chu L, Gu Y. A problem of local cohomology modules. Communications in Algebra 2008; 36 (4):
1603-1607.

\bibitem{jzh} Hassanzadeh SH, Jahangiri M, Zakeri H. Asymptotic behaviour and Artinian property of graded local cohomology modules. Communications in Algebra 2009; 37 (11):
4095-4102. 

\bibitem{hellus2001set} Hellus M. On the set of associated primes of a local cohomology module. Journal of Algebra 2001; 237 (1):
406-419. 

\bibitem{cohomological3} Jahangiri M, Sayyari K. Cohomological dimension with respect to the linked ideals. Journal of Algebra and Its Applications. doi: 10.1142/s0219498821501048.
    

\bibitem{characterization2} Jahangiri M, Sayyari K. Characterization of some special rings via linkage. Journal of Algebra and Related Topics 2020; 8 (1): 67-81. 




\bibitem{linkage1} Jahangiri M, Sayyari K. Linkage of ideals over a module. Journal of Algebraic Systems 2021; 8 (2):
267-279. 

\bibitem{jz} Jahangiri M, Zakeri H. Local cohomology modules with respect to an ideal containing the irrelevant ideal. Journal of Pure and Applied Algebra 2009; 213 (4):
573-581.


\bibitem{kirby1973artinian} Kirby D. Artinian modules and Hilbert polynomials. The Quarterly Journal of Mathematics 1973; 24 (1):
47-57. 

\bibitem{thomasmarley} Marley T. Finitely graded local cohomology and the depths of graded algebras. Proceedings of the American Mathematical Society 1995; 123 (12):
3601-3607. 

\bibitem{matsumura1989commutative} Matsumura H. Commutative ring theory. New Rochelle, NY, USA: Cambridge University Press, 1989.


\bibitem{nagel1994cohomological} Nagel U, Schenzel P. Cohomological annihilators and Castelnuovo-Mumford regularity. Commutative Algebra 1994; 159: 307-328.

\bibitem{peskine1974liaison} Peskine C, Szpiro L. Liaison des vari{\'e}t{\'e}s alg{\'e}briques. I. Inventiones mathematicae 1974; 26 (4): 271-302.

\bibitem{puthenpurakal2017graded} Puthenpurakal T. Graded components of local cohomology modules. accepted for publication in Collectena Math 2020.

\bibitem{rotman-spectral} Rotman JJ. An introduction to homological algebra. London, UK: Academic Press Limited, 1979.


\bibitem{rotthaus2005some} Rotthaus C, {\c{S}}ega LM. Some properties of graded local cohomology modules. Journal of Algebra 2005; 283 (1):
232-247.


\bibitem{singh2007local} Singh AK, Walther U. Local cohomology and pure morphisms. Illinois Journal of Mathematics 2007; 51 (1):
287-298. 



\end{thebibliography}
\end{document}